\documentclass[secthm,seceqn,amsthm,ussrhead,12pt]{amsart}
\usepackage{amsmath,latexsym}
\usepackage[english]{babel}
\usepackage[psamsfonts]{amssymb}
\usepackage{times}
\usepackage{cite}
\usepackage{pdflscape}
\usepackage{ulem}
\usepackage[mathcal]{euscript}
\usepackage{tikz}
\usepackage{hyperref}
\usepackage{cancel}
\usepackage{stmaryrd}
\usetikzlibrary{arrows}

\newtheorem{Th}{Theorem}
\newtheorem{lem}[Th]{Lemma}
\newtheorem{proposition}[Th]{Proposition}

\newtheorem{corollary}[Th]{Corollary}

\begin{document}

{\Large
On Thompson's conjecture  \\ for finite simple exceptional groups of Lie type
\footnote{The work was supported by
RFBR 17-51-04004, BRFFR F17RM-063 and
the President’s Programme ‘Support of Young Russian Scientists’ (grant MK-6118.2016.1).}

}
\medskip

\medskip

\medskip

\medskip

{\bf Ilya Gorshkov$^{a,b}$, Ivan Kaygorodov$^a$, Andrei Kukharev$^c$, Aleksei Shlepkin$^d$ }

\

\

{\tiny
$^a$ Universidade Federal do ABC, Santo Andre, Brazil.

$^b$ Institute of Mathematics and Mechanics, Ekaterinburg, Russia.

$^c$ Vitebsk State University, Vitebsk, Belarus.

$^d$ Siberian Federal University, Krasnoyarsk, Russia.

\smallskip

  E-mail addresses:\smallskip

  Ilya Gorshkov (ilygor8@gmail.com),

  Ivan Kaygorodov (kaygorodov.ivan@gmail.com),

  Andrei Kukharev (kukharev.av@mail.ru),

  Aleksei Shlepkin (shlyopkin@mail.ru).

 }

\

{\it Abstract: Let $G$ be a finite group, $N(G)$ be the set of conjugacy classes of the group $G$. In the present paper it is proved $G\simeq L$ if $N(G)=N(L)$, where $G$ is a finite group with trivial center and $L$ is a finite simple group of exceptional Lie type or Tits group.
}

 \smallskip

\section{Introduction}

Let $G$ be a finite group.
Put $N(G)=\{|g^G|\mid g\in G\}$. In 1987 Thompson posed the following conjecture concerning $N(G)$.
\medskip

\textbf{Thompson's Conjecture (see \cite{Kour}, Question 12.38)}. {\it If $L$ is a finite simple group, $G$ is a finite group with trivial
center, and $N(G)=N(L)$, then $G\simeq L$.}

\medskip

We denote by $\omega(G)$ and $\pi(G)$ the set of orders elements in $G$ and the set of all prime divisors of order of $G$ respectively. The set $\omega(G)$ defines a prime graph $GK(G)$, whose vertex set is $\pi(G)$ and two distinct primes $p, q\in\pi(G)$ are adjacent if $pq\in\omega(G)$. Wilson \cite{Wil} end Kondratiev \cite{Kon} obtained the classification of finite simple groups with disconnected prime graph. Using this deep result Chen \cite{Ch96,Ch99} established the Thompson's conjecture for all finite simple groups witch prime graph have more then two connected components. In particular, he proved the validity of Thompson's conjecture for all finite simple exceptional groups of Lie type, with the exception of $F_4(q)$, where $q$ is odd, $E_7(q)$, $E_6(q)$ and $^2E_6 (q)$. Vasil'ev \cite{Vas} has dealt with Thompson's conjecture for smallest non-abelian simple group with connected prime graph (that is a alternating group $Alt_{10}$) and smallest non-abelian simple group of Lie type with connected prime graph (that is the linear group $L_4(4)$). Developing the methods obtained in the article by Vasil'ev \cite{Vas} and Andjideh \cite{Ah11} in article \cite{Xu14} showed that Thompson's conjecture holds for finite simple exceptional groups of type $E_7(q)$.

The aim of this paper is to prove the following theorem.

\begin{Th}
Let $G$ be a finite group with the trivial center such that $N(G)=N(L)$ and $L\in\{F_4(q)$ {\it for odd} $q, E_6(q), ^2E_6(q)\}$. Then
$G\simeq L$.
\end{Th}

In particular, it follows from the Theorem and earlier results that the Thompson's conjecture is valid for all finite simple exceptional groups of Lie type.

\begin{corollary}
Thompson's conjecture is true for all exceptional  groups of Lie type.
\end{corollary}

\section{Definitions and preliminary results}

Let $s(G)$ denotes the number of the prime graph components of $G$, and $\pi_i$ denotes the set of vertexes of the $i$-th
prime graph component of $G$, $i=1, 2, ... , s(G)$, $T(G)=\{\pi_i(G)|i=1, 2, . . . , s(G)\}$. If $G$ has even order,
then we always assume that $2\in \pi_1$. Denote by $t(G)$ the maximal number of primes in $\pi(G)$ pairwise nonadjacent in GK(G), $\rho(G)$ is some independent set with the maximal number of vertices in $GK(G)$. The rest of the notation is
standard and can be found in \cite{Gore}.

\begin{lem}\cite[Lemma 1.4]{Ch96}\label{order}
Suppose $G$ and $M$ are two finite groups satisfying $s(M)\geq2$, $N(G)=N(M)$, and $Z(G)=Z(M)=1$. Then $|G|=|M|$.
\end{lem}
\begin{lem}\cite[Lemma 1.5]{Ch96}\label{ordercomp}
Suppose $G$ and $M$ are two finite groups satisfying $|G|=|M|$, $N(G)=N(M)$. Then $s(G)=s(M)$ and $T(G)=T(M)$.
\end{lem}
\begin{lem}\cite[Theorem A]{Wil}\label{GK}
If a finite group $G$ has disconnected prime graph, then one of the
following conditions holds:
\begin{enumerate}
\item[(a)]{$s(G)=2$ and $G$ is a Frobenius or 2-Frobenius group;}
\item[(b)]{there is a non-abelian simple group $S$ such that $S\leq G = G/F(G)\leq Aut(S)$, where $F(G)$ is the
maximal normal nilpotent subgroup of $G$; moreover, $F(G)$ and $G/S$ are $\pi_1(G)$-subgroups, $s(S)\geq s(G)$,
and for every $i$ with $2\leq i\leq s(G)$, there is $j$ with $2\leq j\leq s(S)$ such that $\pi_i(G)=\pi_j (S)$.}
\end{enumerate}
\end{lem}

\begin{lem}\cite[Lemma 2.6.]{Ah11}\label{Ah}
Let $G(q)$ be a simple group of Lie type in characteristic $p$, then $|G(q)|<(|G(q)|_p)^3$.
\end{lem}

If $\pi$ is a set of primes, then $n_{\pi}$ denotes the $\pi$-part of $n$, that is, the largest divisor $k$ of $n$ with $\pi(k)\subseteq \pi$ and $n_{\pi'}$ denotes the $\pi'$-part of $n$, that is, the ratio $|n|/n_{\pi}$. If $n$ is a nonzero integer and $r$ is an odd prime with $(r, n)=1$, then $e(r, n)$ denotes the multiplicative order of $n$ modulo $r$. Given an odd integer $n$, we put $e(2, n)=1$ if $n\equiv1(mod \ 4)$, and $e(2,n)=2$ otherwise.

Fix an integer $a$ with $|a|>1$. A prime $r$ is said to be a primitive prime divisor of $a^i-1$ if $e(r, a)=i$. We write $r_i(a)$ to denote some primitive prime divisor of $a^i-1$, if such a prime exists, and $R_i(a)$ to denote the set of all such divisors. Zsigmondy \cite{zs} have proved that primitive prime divisors exist for almost all pairs $(a, i)$.

\begin{lem}\cite{zs}\label{Zsig}
 Let $a$ be an integer and $|a|>1$. For every natural number $i$ the set $R_i(a)$ is nonempty, except for the pairs $(a, i)\in\{(2, 1), (2, 6), (-2,2), (-2,3), (3, 1), (-3, 2)\}$.
\end{lem}

For $i\neq2$ the product of all primitive divisors of $a^i-1$ taken with multiplicities is denoted by $k_i(a)$. Put $k_2(a)=k_1(-a)$.

\begin{lem}\label{primitiv}
If $k_i(a^n)=k_j(a^m)$, then $in=jm$ or $(a,in)\in\{(2, 1), (2, 6), (-2,2), (-2,3), (3, 1), (-3, 2)\}$.
\end{lem}
\begin{proof}
Assume that $in>jm$. It follows from Lemma \ref{Zsig} that there exists a prime number $r\in R_{in}(a)\setminus R_k(a^m)$ for any $km<in$. Since $r\in R_i(a^n)$ and $in>jm$, we see that $r\not\in R_j(a^m)$; a contradiction.
\end{proof}

\begin{lem}\cite{Wil, Kon}\label{F4}
Let $G\simeq F_4(q)$, where $q$ is odd. Then
$$s(G)=2, \pi_1(G)=\pi(q(q^6-1)(q^8-1)), \pi_2(G)=\pi(q^4-q^2+1), t(G)=5, \rho(G)=\{r_3, r_4, r_6, r_8, r_{12}\}.$$
\end{lem}

\begin{lem}\cite{Wil, Kon}\label{E6}
Let $G\simeq E_6(q)$. Then
$$s(G)=2, \pi_1(G)=\pi(q(q^5-1)(q^8-1)(q^{12}-1)), \pi_2(G)=\pi((q^6+q^3+1)/(3,q-1)).$$
If $q=2$ then
$$t(G)=5 \mbox{ and }\rho(G)=\{5,13,17,19,31\} \mbox{ else }t(G)=6 \mbox{ and  }\rho(G)=\{r_4, r_5, r_6, r_8, r_9, r_{12}\}.$$
\end{lem}

\begin{lem}\cite{Wil, Kon}\label{2E6}
Let $G\simeq \ ^2E_6(q)$. Then
$s(G)=2,$ $\pi_1(G)=\pi(q(q^5+1)(q^8-1)(q^{12}-1)),$ $\pi_2(G)=\pi((q^6-q^3+1)/(3,q+1)),$ $t(G)=5,$ $\rho(G)=\{r_4, r_8, r_{10}, r_{12}, r_{18}\}.$
\end{lem}

\section{Proof of Theorem}

The proof is divided into two proposition.

\begin{proposition}
Let $G$ be a finite group with the trivial center such that $N(G)=N(L)$, where $L\simeq F_4(q)$ for odd $q$. Then
$G\simeq L$.
\end{proposition}
\begin{proof}

Let $L\simeq F_4(q)$ and $q=p^n$, where $p$ is an odd prime, $N(G)=N(L)$, $Z(G)=1$. It follows from Lemma \ref{F4} that the prime graph of $L$ has two connected components. From Lemmas \ref{order} and \ref{ordercomp} it follows that $|G|=|L|$ and $T(G)=T(L)$; in particular, $s(G)=2$.

\begin{lem}\label{F4-fr}
The group $G$ is not Frobenius group and not $2$-Frobenius group.
\end{lem}
\begin{proof}
Assume that $G$ is a Frobenius group with the kernel $K$ and a complement $C$. Since $K$ is a nilpotent group, there exists $\beta\in N(G)$ such that $\pi(\beta)=\pi(C)$. The graph $GK(G)$ is not connected, and consequently there are no elements in $G$ of order $tr$, where $t\in\pi(K)$ and $r\in\pi(C)$. Thus, $\pi(K)$ is a connected component of the graph $GK(G)$. There is no number $\alpha\in N(L)$ such that $\pi(\alpha)=\pi_2(G)$. Hence, $K$ is a $\pi_2(G)$-group and $|K|=q^4-q^2 + 1 $. But then $|C|=|G|/(q^4-q^2+1)>|K|$; a contradiction.

Assume that $G=A.B.C$ is a $2$-Frobenius group and the subgroups $A.B$ and $B.C$ are Frobenius groups. In this case there are numbers $\alpha$ and $\beta$ in $N(G)$ such that $\pi(\alpha)=\pi(B)$ and $\pi(\beta)=\pi(A.C)$. The subgraphs $GK(AC)$ and $GK(B)$ are the connected components of the graph $GK(G)$, but there is no $\gamma$ in $N(G)$ such that $\pi(\gamma)=\pi_2(G)$; a contradiction.

\end{proof}

It follows from the Lemmas \ref{GK} and \ref{F4-fr} that $G$ contains a unique non-abelian composition factor $S$ such that there is a normal nilpotent $\pi_1(G)$-subgroup $K$ and $S\leq G/K\leq Aut(S) $. Assume that there exists $t\in\pi (K)\cap\rho(G)$. Let $T$ be a Sylow $t$-subgroup of the group $K$, $R$ a Sylow $ r_{12} $-subgroup of the group $G$. Then, from the Frattini argument and Schur-Zassenhaus theorem, we can assume that $R\leq N_G(T)$. Since $tr_{12} \not\in\omega(G)$, then $T.R$ is a Frobenius group. Thus, $|T|-1$ is divisible by $q^4-q^2+1$; a contradiction.

From the fact that $r_1$ and $r_2$ are not adjacent in $GK(G)$ with $r_{12} $, similarly as above, we get that $\{r_1, r_2\} \setminus \{2\}\cap\Pi(K)=\varnothing$. Thus, $\pi(K)\subseteq\{2, p\}$, $s(S) = 2 $ and $|S|_{\pi_2(S)}=|G|_{\pi_2(G)}=q^4-q^2+1$.

Assume that $S$ is isomorphic to the alternating group $Alt_m$. Then one of the numbers $m, m-1$ or $m-2$ is prime and equal to $q^4-q^2+1$. Therefore, $|S|\geq (q^4-q^2 + 1)!/2>|L|=|G|$; a contradiction.

Analyzing the orders of sporadic groups and Tits groups, it is easy to show that $S$ is not isomorphic a sporadic or a Tits group.

Thus, $S$ is a group of Lie type over a field of the order $u=t^m $. Assume that $t\neq p$. By Lemma \ref{Ah} we have $|S|<|S|_t^3$. From the fact that $\pi(K)\subseteq\{2, p\} $ it follows that
$$|Aut(S)| \geq (q^2-1)(q^6-1)(q^8-1)(q^{12}-1)/(2^6(q+1)^4).$$
The number $t$ divides one of the numbers $(q^2-1), (q^6-1), (q^8-1), (q^{12}-1)$, we have $|S|_t \leq (q^2+q+1)^2 $; a contradiction. Thus, $t=p$.

Assume that $S$ is a group of Lie type and of rank less than $3$. Then
$$|G|_{\pi_2}^3 <| G |_{(\{2, p\}\cap\pi_2(S))'} = | S |_{(\{2, p\} \cap\pi_2(S))'} <|S|_{\pi_2(S)}^3 ;$$ a contradiction.

Assume that $2\in\pi(K)$. Since $K$ is nilpotent and there are no elements of order $2r_ {12}$ in the group $G$, there is a Frobenius group with the kernel $T\leq K$ of order $2^l$ and a complement of order $q^4-q^2+1$. We have $|L|_2=|G|_2\geq|S |_2 2^l\geq 64 (q+\varepsilon1)^4_2 $, where $\varepsilon\in\{+,-\}$. The number $q^4-q^2+1$ divides $|T|-1=2^l-1 $. From the description of the orders of simple groups with a disconnected prime graph (see \cite{Kon} and \cite{Wil}) and the fact that the rank of the group $S$ is greater than $3$ it follows that $|S|_2 \geq 2^6$. Thus, $|T|\leq|K|_2 \leq| q-\varepsilon1|_2^4$. However $2^l-1$ is not divisible by $q^4-q^2+1$ for any $2^l\leq|q-\varepsilon1|_2^4 $; a contradiction. Thus, $K$ is a $p$-group.

Since $K$ is a $p$-group, it follows that $$|G|_{\pi_2(G)}^6<|G/K|_{(\{p \} \cup \pi_2(G))'}=|G|_{(\{p\}\cup\pi_2)'}<|G|_{\pi_2(G)}^7.$$
It follows from the description of the orders of the connected components of simple groups that
$$S\in\{F_4 (u), E_6 (u), \ ^2E_6(u)\}.$$

Assume that $S\in\{E_6 (u), \ ^2E_6(u)\} $, where $u=p^m$. From the fact that $|G|_{\pi_2}=|S|_{\pi_2}$ and Lemma \ref{primitiv} it follows that $m=4n/3$. We obtain $|S|_p=q^{36}=p^{48n}>p^{24n}=|L|_p$; a contradiction. Thus, $S\simeq F_4 (u)$. From the fact that $|G|_{\pi_2(G)}=|S|_{\pi_2(S)}$ it follows that $u=q$, and hence the proposition is proved.
\end{proof}

\begin{proposition}
Let $G$ be a finite group with the trivial center such that $N(G)=N(L)$, where $L\simeq \  ^{\varepsilon}E_6(q)$ and $\varepsilon\in\{1,2\}$. Then $G\simeq L$.
\end{proposition}
\begin{proof}
Let $L\simeq \  ^{\varepsilon} E_6(q)$ anf $q=p^n$, where $p$ is a prime, $\varepsilon\in\{1, 2\}$, $N(G)=N(L)$, $Z(G)=1$. It follows from Lemma \ref{F4} that the prime graph of $L$ has two connected components. From Lemmas \ref{order} and \ref{ordercomp} it follows that $|G|=|L|$ and $T(G)=T(L)$; in particular, $s(G)=2$.

\begin{lem}\label{E6-fr}
The group $G$ is not Frobenius group and not $2$-Frobenius group.
\end{lem}
\begin{proof}
Assume that $G$ is a Frobenius group with the kernel $K$ and a complement $C$. Since $K$ is a nilpotent group, there exists $\beta\in N(G)$ such that $\pi(\beta)=\pi(C)$. The graph $GK(G)$ is not connected, and consequently in $G$ there are no elements of order $tr$, where $t\in\pi(K)$ and $r\in\pi(C)$. Thus $\pi(K)$ is a connected component of the graph $GK(G)$. There is no number $\alpha\in N(G)$ such that $\pi(\alpha)=\pi_2(G)$. Hence, $K$ is a $\pi_2(G)$-group and $|K|=q^6-(-1)^{\varepsilon}q^3+1$. But then $|C|=|G|/q^6-(-1)^{\varepsilon}q^3+1>|K|$; a contradiction.

Assume that $G=A.B.C$ is a $2$-Frobenius group and the subgroups $A.B$ and $B.C$ are Frobenius groups. In this case there are numbers $\alpha$ and $\beta$ in $N(G)$ such that $\pi(\alpha)=\pi(B)$ and $\pi(\beta)=\pi(A.C)$. The subgraphs $GK(AC)$ and $GK(B)$ are the connected components of the graph $GK(G)$, but there is no $\gamma$ in $N(G)$ such that $\pi(\gamma)=\pi_2(G)$; a contradiction.

\end{proof}

It follows from the lemmas \ref{GK} and \ref{E6-fr} that $G$ contains a unique non-abelian composition factor $S$ such that there is a normal nilpotent $\pi_1$-subgroup $K$ and $S\leq G/K\leq Aut(S) $. Assume that there exists $t\in \pi(K)\cap\rho(G) $. Let $T$ be a Sylow $t$-subgroup of the group $K$, $R$ a Sylow $r_{9^{\varepsilon}}$-subgroup of the group $G$. Then, from the Frattini argument, we can assume that $R\in N(T)$. Since $tr\not\in\omega(G)$, then $T.R$ is a Frobenius group. Since $tr_{9^{\varepsilon}}\not\in\omega(G)$, then $T.R$ is a Frobenius group. Thus $|T|-1$ is divisible by $q^6-(- 1)^{\varepsilon}q^3+1 $; a contradiction.
From the fact that $r_1$ and $r_2$ are not adjacent in $GK(G)$ with $r_{9^{\varepsilon}}$, similarly as above, we get that $\{r_1, r_2\}\setminus\{2\} \cap \pi(K)=\varnothing$.
Thus,
$$\pi(K)\subseteq\{2, p\}, s(S)=2 \mbox{ and }|S|_{\pi_2(G)}=|G|_{\pi_2(G)}=q^6-(-1)^{\varepsilon} q^3+1 .$$

Assume that $S$ is isomorphic to the alternating group $Alt_m$, then one of the numbers $m, m-1, m-2$ is prime and equal to $q^6-(-1)^\varepsilon q^3+1 $. Therefore, $|S|\geq (q^6-(-1)^\varepsilon q^3 + 1)!> |L|=|G|$; a contradiction.

Analyzing the orders of sporadic groups and Tits groups, it is easy to show that $S$ is not a sporadic group or a Tits group.

Thus, $S$ is a group of Lie type over a field of the order $u=t^m$. Assume that $t\neq p$. By Lemma \ref{Ah} we have $|S|<|S|_t^3$. From the fact that $\pi(K)\subseteq\{2, p\}$ it follows that $$|Aut(S)|\geq(q^2-1)(q^5+(-1)^\varepsilon)(q^6-1)(q^8-1)(q^9+(-1)^\varepsilon)(q^{12}-1)/(2^6(q-1)^6).$$
The number $t$ divides one of the numbers $q^2-1, q^5+(-1)^{\varepsilon}, q^6-1, q^8-1, q^9 + (-1)^{\varepsilon}, q^{12}-1$, we have $|S|_t\leq(q^6-(-1)^{\varepsilon}q^3+1)^3$; a contradiction. Thus, $t=p$.

Assume that $S$ is a group of Lie type and of rank less than $3$. Then
$$|G|_{\pi_2(G)}^3 <|G|_{(\{2, p\}\cap\pi_2(G))'}=|S|_{(\{2,p\}\cap \pi_2 (G))'}<|S|_{\pi_2(G)}^3 ;$$
a contradiction.

Assume that $2\in\pi(K)$ and $p\neq2$. Since $K$ is nilpotent and there are no elements of order $2r_{9^{\varepsilon}}$ in $G$, then in $G$ there is a Frobenius group with the kernel $T\leq K$ of order $2^l$ and a complement of order $q^6-(-1)^{\varepsilon}q^3+1$. We have $|G|_2=|L|_2=64(q+\delta1)^6_2$, where $\delta \in \{+,-\}$. From the description of the orders of simple groups with a disjoint prime graph it follows that $|S|_2 \geq 2^6 $ or that the rank of the group $S$ is less than $3$. Thus, $|K|_2 \leq |q-\delta1|_2^6$. But $2^l-1$ is not divisible by $q^6-(-1)^{\varepsilon}q^3+1$ for any $2^l\leq|q-\delta1|_2^6$; a contradiction. Thus, $K$ is a $p$-group.

Since $K$ is a $p$-group, it follows that
$$|G|_{\pi_2(G)}^6 <|G/K|_{(\{p,\}\cup\pi_2(G))'}=|G|_{\{p,\pi_2(G)\}'}<|G|_{\pi_2(G)}^7.$$
From the description of the orders of the connected components of simple groups it follows that this condition is satisfied by $F_4(q), E_6(u), ^2E_6(u) $.

Assume that $S\simeq F_4(u)$, where $u=p^m$. From the fact that $|S|_{\pi_2(G)}=|G|_{\pi_2(G)}$ and the Lemma \ref{primitiv} it follows that $m=3n/4$. We obtain $|Aut(S)|_{p'}>|L|_{p'}$; a contradiction. Thus $S\not\simeq F_4(q)$.

Assume that $S\simeq \ ^{\alpha}E_6(u)$, where $\alpha \in \{1,2\}\setminus\{\varepsilon\}$. We have $p^{6n}-p^{3n} + 1 = p^{6m}+p^{3m}+1$; contradiction. The proposition is proved.

\end{proof}

\end{document}